\documentclass[11pt,a4paper]{article}
\usepackage[T1]{fontenc}
\usepackage[utf8]{inputenc}
\usepackage{lmodern}
\usepackage{amsmath,amssymb,amsthm,mathtools}
\usepackage[margin=28mm]{geometry}
\usepackage{pgfplots}
\usepackage{microtype}
\usepackage{hyperref}
\hypersetup{colorlinks=true,linkcolor=black,citecolor=black,urlcolor=black,
	pdftitle={Entire p-harmonic functions with an isolated critical point},
	pdfauthor={xxxxx}}
\usepackage{xcolor}

\newtheorem{theorem}{Theorem}[section]
\newtheorem{proposition}[theorem]{Proposition}
\newtheorem{lemma}[theorem]{Lemma}
\newtheorem{corollary}[theorem]{Corollary}
\theoremstyle{remark}
\newtheorem{remark}[theorem]{Remark}
\numberwithin{equation}{section}
\newcommand{\R}{\mathbf R}
\newcommand{\Sph}{\mathbf S^{n-1}}
\newcommand{\avgint}{\mathchoice
	{\AvgInt\displaystyle\textstyle{-}}
	{\AvgInt\textstyle\scriptstyle{-}}
	{\AvgInt\scriptstyle\scriptscriptstyle{-}}
	{\AvgInt\scriptscriptstyle\scriptscriptstyle{-}}\!\int}
\newcommand{\AvgInt}[3]{%
	\setbox0=\hbox{$#1{#2#3}{\int}$}%
	\vcenter{\hbox{$#2#3$}}\kern-.5\wd0}
\newcommand{\dd}{\,d}
\DeclareMathOperator{\diver}{div}
\title{Entire $p$-harmonic functions with an isolated critical point\\
	and failure of  $C^1$-regularity of the natural gradient}

\date{}

\author{Erno Kauranen and Jarkko Siltakoski}

\begin{document}
	\maketitle
	
	\begin{abstract}
		For every $n\ge2$ and $1<p<\infty$, we show an existence of a nonconstant homogeneous
		entire $p$-harmonic function whose only critical point is the origin.
		The angular parts of these functions are axially symmetric and even across the equator. To establish the solution we use contraction argument at the pole in conjuction with a shooting argument to the equator. It follows that in every dimension $n\ge3$ and for $p<2$ sufficiently
		close to $2$, the natural gradient has pointwise H\"older
		exponent strictly below one at the origin. Such examples settle the remaining scalar case of the conjecture of Balci, Diening, and Weimar by disproving both the $C^1$ assertion and the linear $L^2$ mean oscillation estimate.
	\end{abstract}
	
	\section{Introduction}
	
	Let   $n\ge1$ be an integer and $1<p<\infty$. A function
	$u\in W^{1,p}_{\mathrm{loc}}(\R^n)$ is called \emph{$p$-harmonic} if it is a weak solution to $\mathrm{div} (|\nabla u|^{p-2} \nabla u)= 0$ that is, the equation
	\begin{equation}\label{eq:weak}
		\int_{\R^n}\langle|\nabla u|^{p-2}\nabla u,\nabla\varphi\rangle\dd x=0
	\end{equation} holds for every $\varphi\in C_0^\infty(\R^n)$.
	For background on weak solutions and the regularity theory
	of the $p$-Laplace equation, we refer to Lindqvist's notes
	\cite{Lindqvist2006}.
	In this paper we denote
	\begin{equation*}
		V_p(\nabla u):=|\nabla u|^{\frac{p-2}{2}}\nabla u,
		\qquad V_p(0):=0
	\end{equation*}
	which is the natural gradient associated with the $p$-energy.

	Homogeneous (or quasiradial) solutions form a natural class for the
	$p$-Laplace equation. In fact, Kr\'ol's construction and its later
	developments give homogeneous $p$-harmonic functions of different
	orders in $\mathbf{R}^2$, see \cite{Krol1973,Aronsson1988}. It is also known that the critical
	points of a nonconstant planar $p$-harmonic function forms a discrete set
	\cite{Manfredi1988} and even the precise weak regularity of such $p$-harmonic functions is known \cite{IwaniecManfredi1989} from which it follows that natural gradient is of class $C^{1}$ in the plane. Moreover, separation of variables and the resulting nonlinear
	equation on the sphere have been considered in cones and on manifolds,
	see \cite{PorrettaVeron2009, Tolksdorf1983} and the references therein. Constructions based on isoparametric polynomials can be found in \cite{Tkachev2020}.
	
	The higher-dimensional theory is less understood. Although planar solutions admit cylindrical extensions to $\R^n$ such extensions are not genuinely higher dimensional. In particular, they do not have the isolated critical point property studied here since their critical sets contain an $(n-2)$-dimensional subspace. We overcome this limitation by constructing homogeneous solutions whose only critical point is the origin. Our main theorem is the following.

	\begin{theorem}\label{thm:entire}
		For every $n\ge2$ and $1<p<\infty$, there are $k>1$ and
		$f\in C^2([0,\pi])$, with $f(0)=1$, such that
		\begin{equation}\label{eq:ansatz}
			u(x)=|x|^k f(\theta),\qquad
			\theta=\arccos\frac{x_n}{|x|}\quad(x\ne0),\qquad u(0)=0,
		\end{equation}
		is a $p$-harmonic function in
		$C^1(\R^n)\cap C^2(\R^n\setminus\{0\})$ such that its only critical point is the origin.
		Moreover, $f'(0)=f'(\pi)=0$ and $f(\pi-\theta)\equiv f(\theta)$.
	\end{theorem}
	
	The assumption $n\ge2$ is necessary because in one dimension every solution of the equation is affine. In higher dimensions, we must prevent the radial and angular derivatives from vanishing simultaneously. For our ansatz, this means that the angular factor and its derivative must have no common zero. We express these two quantities using an amplitude and a phase. We first solve the phase equation and then obtain the amplitude as the exponential of a finite integral, which makes it strictly positive. Since the angular factor is proportional to the cosine of the phase and its derivative to the sine, they cannot both vanish. Finally, we choose the homogeneity exponent so that the angular derivative vanishes at the equator. This allows us to reflect the solution from the upper hemisphere to the whole sphere while keeping the gradient nonzero away from the origin. For a shooting construction of positive homogeneous solutions
in a half space, see also \cite{LlorenteManfrediTroyWu2019}.

	The construction has a further consequence for the regularity of the natural gradient. To put this in context,  Balci, Diening, and Weimar conjectured that the natural gradient of a
	$p$-harmonic map belongs to $C^1$ and satisfies a linear $L^2$
	mean oscillation decay estimate
	\cite[Conjecture~2.28]{BalciDieningWeimar2020}. Recently, Balci, Behn,
	Diening, and Storn constructed homogeneous $p$-harmonic maps
	$U:\R^n\to\R^N$ for every $n\geq2$ and suitable $N>1$, see \cite{BalciBehnDieningStorn2026}.
	Their construction starts from harmonic homogeneous polynomial maps $h$
	satisfying $|h(x)|=|x|^k$ and produces maps of the form
	$U(x)=|x|^{\gamma-k}h(x)$. For $k\geq2$, one has
	$\gamma>1$ and it follows that
	$\nabla U(x)\neq0$ for $x\neq0$, while $\nabla U(0)=0$. These examples
	yield new upper bounds for the regularity of vectorial $p$-harmonic maps and,
	when $1<p<2$ and $n\geq3$, disprove the conjectured $C^1$ regularity of
	$V_p(\nabla U)$. Since their ansatz requires $N>1$ it does not cover the scalar case. The following corollary disproves the scalar version of the conjecture in every dimension $n \geq 3$, for $p<2$ sufficiently close to $2$.

\begin{corollary}\label{thm:counterexample}
	For every $n\ge3$ there is $\delta_n\in(0,1)$ such that for every
	$p\in(2-\delta_n,2)$ there exists a $p$-harmonic function
	$u_p\in C^2(\R^n)$ whose unique critical
	point is the origin, and there are
	$\beta_n(p)\in(0,1)$ and constants $0<c\le C<\infty$ such that
	\begin{equation}\label{eq:pointwise}
		c|x|^{\beta_n(p)}
		\le |V_p(\nabla u_p)(x)|
		\le C|x|^{\beta_n(p)}
	\end{equation}
	for every $x\in\R^n$.
	In particular, $V_p(\nabla u_p)$ is not locally Lipschitz at the origin.
	Moreover, for every fixed $r>0$ there is no finite constant $C_0$ such that
	\begin{equation}\label{eq:failure-linear-decay}
		\frac{\left(
			\avgint_{B_{\theta r}}
			\left|
			V_p(\nabla u_p)
			-
			\avgint_{B_{\theta r}}V_p(\nabla u_p)\dd x
			\right|^2
			\dd x
			\right)^{1/2} }{\left(
			\avgint_{B_r}
			\left|
			V_p(\nabla u_p)
			-
			\avgint_{B_r}V_p(\nabla u_p)\dd x
			\right|^2
			\dd x
			\right)^{1/2} }
		\le
		C_0\theta
	\end{equation}
	holds for every $\theta\in(0,1]$.
\end{corollary}

The paper is organized as follows. Section 2 derives the phase function construction
and in Section 3 we reflect the solution to the whole sphere. In the last section we prove the failure of the conjecture related to the natural gradient. Finally, the coordinate reductions are proved in the Appendix.

\newpage
\section{The construction in the upper hemisphere}

Our aim is to show a existence of certain homogenous $p$-harmonic functions and it is therefore convenient to seek for such solutions with the aid of the spherical $p$-Laplace equation. The objective is to first construct a desired function in the upper hemisphere and then reflect it across the equator.

To this end fix $n\ge2$ and put $L=\pi/2$. We first construct angular functions on
$[0,L]$ for every $p>1$ and $k\ge1$. The choice of $k$ and the extension
to the whole sphere are made in Section~\ref{sec:reflection}.
For the homogeneous ansatz \eqref{eq:ansatz} for the $p$-Laplace equation, set
\begin{equation}\label{eq:Q-Lambda}
 Q=k^2f^2+(f')^2,\qquad \Lambda(p,k)=k\big((p-1)k+n-p\big).
\end{equation}
Then (assuming that  $Q>0$) the $p$-Laplace equation becomes
\begin{equation}\label{eq:angular}
 \big(\sin^{n-2}\theta\,Q^{(p-2)/2}f'\big)'
 +\Lambda(p,k)\sin^{n-2}\theta\,Q^{(p-2)/2}f=0, \quad 0<\theta<\pi.
\end{equation}
The computation is given in the Appendix.
We now seek an amplitude $\rho>0$ and a phase
$\psi$, with $\rho(0)=1$ and $\psi(0)=0$ such that
\begin{equation}\label{eq:polar-data}
 f_{p,k}=\rho\cos\psi\qquad\text{and}\qquad f_{p,k}'=-k\rho\sin\psi.
\end{equation}
Suppressing the parameters when convenient, the resulting equations are (see Appendix for derivation)
\begin{align}
 \psi'&=F(\theta,\psi;p,k)
 =\frac{(p-1)k+(n-p)\cos^2\psi
 -(n-2)\cot\theta\sin\psi\cos\psi}{D_p(\psi)},\label{eq:phase}\\
 \frac{\rho'}{\rho}&=H(\theta,\psi;p,k)
 =\frac{((p-2)k+n-p)\sin\psi\cos\psi
 -(n-2)\cot\theta\sin^2\psi}{D_p(\psi)},\label{eq:amplitude}
\end{align}
where $D_p(v)=\cos^2v+(p-1)\sin^2v$.
In particular, $\min\{1,p-1\}\le D_p(v)\le\max\{1,p-1\}$.
We separate the phase construction, its endpoint properties, and the
recovery of the angular function into three lemmas.

\begin{lemma}[The phase on the upper hemisphere]\label{lem:pole}
 For every $p>1$ and $k\ge1$, there is a unique
 $\psi(\,\cdot\,;p,k)\in C^1([0,L])$ solving \eqref{eq:phase} on $(0,L]$
 and satisfying
 \begin{equation}\label{eq:pole-expansion}
  \psi(\theta;p,k)=z_0(p,k)\theta+O(\theta^3),\qquad
  z_0(p,k)=\frac{(p-1)k+n-p}{n-1}
  \quad(\theta\to0).
 \end{equation}
 Moreover, $\psi'=z_0+O(\theta^2)$ at zero, and the phase depends
 continuously on $(p,k)$ in $C^1([0,L])$.
\end{lemma}

\begin{proof}
 Write $\lambda=(p,k)$ and allow $\lambda$ to range over the open set
 $\mathcal P=(1,\infty)\times\R$.
 Define $\sigma(t)=\sin t/t$ for $t\ne0$, with $\sigma(0)=1$, and
 $\chi(\theta)=\cos\theta/\sigma(\theta)$ for $|\theta|<\pi$.
 Both functions are smooth and even, and $\chi(\theta)=\theta\cot\theta$
 for $\theta\ne0$. The rescaled coefficient has the smooth extension
 \[
 \widetilde F(\theta,\xi;\lambda)
 =\frac{(p-1)k+(n-p)\cos^2(\theta\xi)
 -(n-2)\chi(\theta)\xi\sigma(\theta\xi)\cos(\theta\xi)}
 {D_p(\theta\xi)}.
 \]
 For $\theta>0$ this equals $F(\theta,\theta\xi;\lambda)$.
 It is even in $\theta$, and its value at zero is
 $(n-1)z_0(\lambda)-(n-2)\xi$. Here it is the rescaled coefficient
 that extends to zero. Taylor's formula gives
 \begin{equation}\label{eq:pole-new-rescaling}
 \widetilde F(\theta,\xi;\lambda)
 =(n-1)z_0(\lambda)-(n-2)\xi+\theta^2G(\theta,\xi;\lambda),
 \end{equation}
 where $G(\theta,\xi;\lambda)=\int_0^1(1-s)
 \partial_1^2\widetilde F(s\theta,\xi;\lambda)\,ds$ is smooth. The symbol $\partial_1$ denotes differentiation in the first argument.

 Fix $\lambda_*\in\mathcal P$ and a compact rectangle
 $K\subset\mathcal P$ with $\lambda_*$ in its interior. Put $a=L/2$ and
 $E=\{(\theta,\xi,\lambda):0\le\theta\le a,\ \lambda\in K, \ |\xi-z_0(\lambda)|\le1\}$.
 Let $M=\max_E|G|$ and $M_1=\max_E|\partial_\xi G|$.
 Choose $0<\delta\le a$ so that $\delta^2M/(n+1)\le1/2$ and
 $q:=\delta^2M_1/(n+1)\le1/2$. On the closed unit ball $B$ of
 $C([0,\delta])$, define
 $(T_\lambda v)(\theta)=\theta^2\int_0^1s^n G\bigl(\theta s,z_0(\lambda)+v(\theta s);\lambda\bigr)\,ds$.
 Then $\|T_\lambda v\|_\infty\le1/2$ and
 $\|T_\lambda v-T_\lambda\widehat v\|_\infty
 \le q\|v-\widehat v\|_\infty$. Thus by the Banach fixed point theorem the operator $T_\lambda$ has a unique fixed
 point $v_\lambda\in B$. Setting $z_\lambda(\theta)=z_0(\lambda)+v_\lambda(\theta)$
we get 
 \begin{equation}\label{eq:pole-new-fixedpoint}
 z_\lambda(\theta)-z_0(\lambda)
 =\theta^2\int_0^1s^nG(\theta s,z_\lambda(\theta s);\lambda)\,ds
 \end{equation}
and in particular $|z_\lambda(\theta)-z_0(\lambda)|\le M\theta^2/(n+1)$.
For $\theta>0$, substitute $t=\theta s$ in the integral to rewrite \eqref{eq:pole-new-fixedpoint} as
\begin{equation*}
z_\lambda(\theta) - z_0 (\lambda) = \theta^{1-n}\int_0^\theta t^nG(t,z_\lambda(t);\lambda)\,dt.
\end{equation*}
This immediately yields $z_\lambda \in C^1([0, \delta])$ with $z_\lambda ^\prime(0)=0$. \begin{equation}\label{eq:pole-new-z-equation}
\theta z_\lambda'(\theta)+(n-1)(z_\lambda(\theta)-z_0(\lambda))
=\theta^2G(\theta,z_\lambda(\theta);\lambda).
\end{equation}
Define $\psi_\lambda(\theta)=\theta z_\lambda(\theta)$ in $[0, \delta]$.
Combining \eqref{eq:pole-new-rescaling} and
\eqref{eq:pole-new-z-equation} yields
\begin{equation}\label{eq:pole-new-local-phase}
\psi_\lambda'(\theta)
=z_0(\lambda)-(n-2)(z_\lambda(\theta)-z_0(\lambda))
+\theta^2G(\theta,z_\lambda(\theta);\lambda)
=\widetilde F(\theta,z_\lambda(\theta);\lambda).
\end{equation}
 This proves the phase equation and both expansions at zero, with
 remainders uniform for $\lambda\in K$.

 For uniqueness, let $\varphi$ be another phase with the stated
 expansion and put $\zeta(\theta)=\varphi(\theta)/\theta$ for $\theta>0$,
 with $\zeta(0)=z_0(\lambda)$. Then $\zeta-z_0(\lambda)=O(\theta^2)$
 and \eqref{eq:pole-new-rescaling} gives
 $$\theta\zeta'(\theta)+(n-1)(\zeta(\theta)-z_0(\lambda))=\theta^2G(\theta,\zeta(\theta);\lambda).$$
 Multiplication by $\theta^{n-2}$ and integration from zero give
 \eqref{eq:pole-new-fixedpoint} for $\zeta$, because the boundary term
 $\theta^{n-1}(\zeta-z_0(\lambda))$ tends to zero.
 On a sufficiently small interval $[0,\varepsilon]\subset[0,\delta]$,
 both $\zeta-z_0(\lambda)$ and $z_\lambda-z_0(\lambda)$ lie in the
 closed unit ball. The same contraction estimate, with coefficient
 $\varepsilon^2M_1/(n+1)<1$, makes them equal there. Uniqueness for the
 regular differential equation then gives equality on their common
 interval away from zero.

The same $\delta$ and $q$ work for all $\lambda\in K$. For
 $\lambda,\mu\in K$, the fixed point equations implies that
 $$\|v_\lambda-v_\mu\|_\infty \le\frac{1}{1-q}\|T_\lambda v_\mu-T_\mu v_\mu\|_\infty$$ and the right-hand side converges to zero as $\lambda\to\mu$
 by uniform continuity of $G$ on $E$. Thus $\lambda\mapsto z_\lambda$
 is continuous into $C([0,\delta])$. The identity
 $\psi_\lambda=\theta z_\lambda$ and
 \eqref{eq:pole-new-local-phase} give continuity of the phase in
 $C^1([0,\delta])$.

 On $[\delta,L]\times\R\times K$, the coefficients $F$ and
 $\partial_vF$ are bounded, since they are smooth and $\pi$ periodic
 in $v$. Write their bounds as $C_0$ and $C_1$, respectively.
 The solution starting from $\psi_\lambda(\delta)$ extends to $L$:
 the bound $|\psi_\lambda'|\le C_0$ gives a finite limit at any finite
 endpoint, where the usual existence theorem continues the solution.
 To check parameter continuity, put
 $\omega(\lambda,\mu)=\max_{[\delta,L]\times[0,\pi]}
 |F(\theta,v;\lambda)-F(\theta,v;\mu)|$. This tends to zero as
 $\lambda\to\mu$. Gronwall's inequality gives
 $$\|\psi_\lambda-\psi_\mu\|_{C([\delta,L])} \le e^{C_1(L-\delta)} \bigl(|\psi_\lambda(\delta)-\psi_\mu(\delta)| +(L-\delta)\omega(\lambda,\mu)\bigr).$$
 The equations also bound the norm of the derivative difference by
 $C_1\|\psi_\lambda-\psi_\mu\|_{C([\delta,L])}+\omega(\lambda,\mu)$.
 This proves continuity in $C^1([0,L])$. Local uniqueness at zero and
 regular uniqueness away from zero show that all choices of $K$ and
 $\delta$ give the same phase. Since $\lambda_*$ was arbitrary, the phases define a continuous map
$\mathcal P\to C^1([0,L])$.
\end{proof}

\begin{lemma}[The endpoint map]\label{lem:shooting}
 The map $\Psi(p,k)=\psi(L;p,k)$ is continuous for $p>1$ and $k\ge1$.
 For each fixed $p>1$ it is strictly increasing in $k$, and
 \begin{equation}\label{eq:shooting-properties}
  \Psi(p,1)=\frac\pi2,\qquad \lim_{k\to\infty}\Psi(p,k)=+\infty.
 \end{equation}
\end{lemma}

\begin{proof}
 Continuity follows from Lemma~\ref{lem:pole}.
 Fix $p>1$ and $k_2>k_1\ge1$, and let
 $d(\theta)=\psi(\theta;p,k_2)-\psi(\theta;p,k_1)$.
 By \eqref{eq:pole-expansion},
 $d(\theta)=\frac{p-1}{n-1}(k_2-k_1)\theta+O(\theta^3)>0$
 for small positive $\theta$. If $d$ had a first zero
 $\theta_*\in(0,L]$, its left derivative there would be nonpositive.
 But the two phases have the same value there, so \eqref{eq:phase} gives
 $d'(\theta_*)=(p-1)(k_2-k_1)/D_p(\psi(\theta_*;p,k_1))>0$.
 This contradiction proves strict monotonicity, including at $L$.

 For $k=1$, direct substitution gives $F(\theta,\theta;p,1)=1$ and
 $z_0(p,1)=1$. Lemma~\ref{lem:pole} therefore gives
 $\psi(\theta;p,1)=\theta$. In particular, $\Psi(p,1)=L$ and
 $\psi(\theta;p,k)\ge\theta$ for $k\ge1$.
 The value $k=1$ provides a known lower endpoint for the shooting
 argument in Section \ref{sec:reflection}.

 For fixed $p$, put $N_p(k)=(p-1)k-|n-p|-(n-2)/2$ and
 $d_p=\max\{1,p-1\}$. On $[L/2,L]=[\pi/4,\pi/2]$,
 $0\le\cot\theta\le1$ and $|\sin v\cos v|\le1/2$, so the numerator
 of $F(\theta,v;p,k)$ is at least $N_p(k)$. For sufficiently large
 $k$ this lower bound is nonnegative, and $D_p(v)\le d_p$ gives
 $\psi'(\theta;p,k)\ge N_p(k)/d_p$. Since $\psi(L/2;p,k)\ge L/2$,
 integration yields $\Psi(p,k)\ge (L/2)N_p(k)/d_p\to\infty$.
\end{proof}

\begin{lemma}[The amplitude and angular function]\label{lem:amplitude}
 For the phase in Lemma~\ref{lem:pole}, define
 \begin{equation}\label{eq:profile}
 \rho(\theta;p,k)=\exp\!\left(\int_0^\theta
 H(s,\psi(s;p,k);p,k)\,ds\right),\qquad
 f_{p,k}=\rho\cos\psi.
 \end{equation}
 The integrand extends continuously to zero with value zero.
 Then $\rho\in C^1([0,L])$ is positive, $\rho(0)=1$,
 $f_{p,k}\in C^2([0,L])$, $f_{p,k}(0)=1$, and $f_{p,k}'(0)=0$.
 The identities \eqref{eq:polar-data} hold, $f_{p,k}$ solves
 \eqref{eq:angular} on $(0,L)$, and
 \begin{equation}\label{eq:positive-Q}
 Q_{p,k}=k^2f_{p,k}^2+(f_{p,k}')^2=k^2\rho^2>0
 \qquad\text{on }[0,L].
 \end{equation}
 Moreover, $\rho$ and $f_{p,k}$ depend continuously on $(p,k)$ in
 $C^1([0,L])$ and $C^2([0,L])$, respectively.
\end{lemma}

\begin{proof}
 Use $\lambda=(p,k)$ and the functions $\sigma$, $\chi$, and $z_\lambda$
 from the proof of Lemma~\ref{lem:pole}. The smooth function
 \[
 J(\theta,\xi;\lambda)=
 \frac{((p-2)k+n-p)\xi\sigma(\theta\xi)\cos(\theta\xi)
 -(n-2)\chi(\theta)\xi^2\sigma(\theta\xi)^2}{D_p(\theta\xi)}
 \]
 satisfies $H(\theta,\theta\xi;\lambda)=\theta J(\theta,\xi;\lambda)$
 for $\theta>0$. Thus $h_\lambda(\theta)=H(\theta,\psi_\lambda(\theta);\lambda)$
 extends to zero by $h_\lambda(0)=0$, since near zero it equals
 $\theta J(\theta,z_\lambda(\theta);\lambda)=O(\theta)$.
 Extend $z_\lambda$ to $[0,L]$ by
\[
z_\lambda(\theta)=\int_0^1\psi_\lambda'(t\theta)\,dt.
\]
Since $\psi_\lambda(0)=0$ and $\psi_\lambda'(0)=z_0(\lambda)$,
this agrees with $\psi_\lambda(\theta)/\theta$ for $\theta>0$
and with $z_0(\lambda)$ at zero. Moreover,
\[
\|z_\lambda-z_\mu\|_{C([0,L])}
\le \|\psi_\lambda'-\psi_\mu'\|_{C([0,L])},
\]
so Lemma~\ref{lem:pole} gives continuity of
$\lambda\mapsto z_\lambda$ into $C([0,L])$. The identity
\[
h_\lambda(\theta)
=\theta J(\theta,z_\lambda(\theta);\lambda),
\qquad 0\le\theta\le L,
\]
and uniform continuity of $J$ on compact sets therefore imply
that $\lambda\mapsto h_\lambda$ is continuous into $C([0,L])$.
 Formula \eqref{eq:profile} now gives $\rho>0$, $\rho(0)=1$, and
 $\rho'=h_\lambda\rho$, with continuous parameter dependence in
 $C^1([0,L])$. The coefficient formulas give
 $$H(\theta,v;\lambda)\cos v-F(\theta,v;\lambda)\sin v=-k\sin v.$$
 Differentiating $f_{p,k}=\rho\cos\psi$ therefore proves
 $f_{p,k}'=-k\rho\sin\psi$, first for $\theta>0$ and then at zero
 by continuity. The right side belongs to $C^1([0,L])$, and hence
 $f_{p,k}''=-k\rho'\sin\psi-k\rho\cos\psi\,\psi'$.
 In particular, $f_{p,k}''(0)=-kz_0(p,k)$. The formulas for
 $f_{p,k}$ and its first two derivatives give its continuous parameter
 dependence in $C^2([0,L])$. Finally, \eqref{eq:positive-Q} follows
 from \eqref{eq:polar-data}, and the converse computation in
 Appendix proves \eqref{eq:angular} on $(0,L)$.
\end{proof}

\section{The reflection and the entire solution}\label{sec:reflection}
We are now ready to prove the main theorem.
\begin{proof}[Proof of Theorem~\ref{thm:entire}]
	By Lemma~\ref{lem:shooting}, the continuous map $k\mapsto\Psi(p,k)$
 starts at $\Psi(p,1)=\pi/2<\pi$ and tends to infinity.
 The intermediate value theorem gives a root, and strict monotonicity
 makes it unique. Thus there is a unique $k>1$ with
	\begin{equation}\label{eq:root}
		\Psi(p,k)=\pi.
	\end{equation}
	Set $L=\pi/2$ and take $\rho$ and $f=f_{p,k}$ from Lemma~\ref{lem:amplitude}.
	Then $\rho>0$, $f\in C^2([0,L])$, $f(0)=1$, and $f'(L)=0$
	by \eqref{eq:polar-data}.
	
	Extend $f$ to $[0,\pi]$ by
	\begin{equation}\label{eq:reflection}
		f(\pi-\theta)=f(\theta).
	\end{equation}
	Since $f'(L)=0$, the extension is $C^2$ across $L$.
	Equation \eqref{eq:angular} is invariant under this reflection, so it
	holds on $(0,\pi)$. Also, $Q$ stays positive by
	\eqref{eq:positive-Q} and reflection.
	
	We verify the regularity at the poles using an elementary fact:
	if $h\in C^2([0,a))$ and $h'(0)=0$, then $y\mapsto h(|y|)$ is
	$C^2$ near zero. Indeed, for $s=|y|>0$,
	\begin{equation}\label{eq:radial-hessian}
		D_y^2(h(|y|))=\frac{h'(s)}s I+
		\left(h''(s)-\frac{h'(s)}s\right)\frac{y\otimes y}{s^2}.
	\end{equation}
	Since $h'(s)/s\to h''(0)$, this Hessian tends to $h''(0)I$.
	Also $\nabla(h(|y|))=h''(0)y+o(|y|)$, which proves
	differentiability of the gradient at zero.
	Near the north pole, use the chart
	$\omega=(y,\sqrt{1-|y|^2})$, with $y\in\R^{n-1}$ small,
	and apply this fact to $h(s)=f(\arcsin s)$.
	This also applies when $n=2$, and hence
	$\omega\mapsto f(\arccos\omega_n)$ is $C^2$ on $\Sph$.

	By compactness there are constants $0<m\le M<\infty$ with
	$m\le Q\le M$ on $[0,\pi]$. The gradient formula from
	Appendix gives
	\begin{equation}\label{eq:gradient}
		\nabla u(r\omega)=r^{k-1}\big(kf(\theta)\omega+f'(\theta)e_\theta\big),
		\qquad |\nabla u(r\omega)|=r^{k-1}\sqrt{Q(\theta)}.
	\end{equation}
	Here $\theta=\arccos\omega_n$, and $e_\theta$ is the unit tangent
	in the direction of increasing $\theta$ away from the poles; the term
	$f'(\theta)e_\theta$ extends continuously by zero at the poles.
	Thus $u\in C^2(\R^n\setminus\{0\})$ and it has no critical point
	there. The equation holds classically away from the axis by
	\eqref{eq:angular}. Since the gradient is nonzero everywhere off zero,
	the flux $\mathcal A=|\nabla u|^{p-2}\nabla u$ is $C^1$ there.
	Its divergence is continuous and vanishes on the complement of the
	axis, a dense set. It therefore vanishes throughout
	$\R^n\setminus\{0\}$. In particular, the equation holds weakly on
	that punctured space.
	
	The estimates $|u(x)|\le C|x|^k$ and
	$|\nabla u(x)|\le C|x|^{k-1}$, with $k>1$, show that the extension
	$u(0)=0$ belongs to $C^1(\R^n)$, with $\nabla u(0)=0$.
	In particular, $u\in W^{1,p}_{\mathrm{loc}}(\R^n)$.
	
	To remove zero from the weak equation, take
	$\varphi\in C_0^\infty(\R^n)$ and integrate by parts outside
	$B_\varepsilon$. Since $\diver\mathcal A=0$ there it follows that
	\[
	\left|\int_{\R^n\setminus B_\varepsilon}
	\langle\mathcal A,\nabla\varphi\rangle\dd x\right|
	\le \|\varphi\|_\infty\int_{\partial B_\varepsilon}|\mathcal A|\dd S
	\le C\varepsilon^{n-1+(k-1)(p-1)}\longrightarrow0.
	\]
	Local integrability of $\mathcal A$ lets us pass to the limit,
	proving \eqref{eq:weak} and completing the construction.
\end{proof}
\begin{remark}
For each fixed $n\ge2$ and $p>1$, the same construction gives
infinitely many entire $p$-harmonic functions whose only critical
point is the origin. Indeed, choosing $m\pi$, $m \in \mathbf N \setminus{\{0\}}$, in \eqref{eq:root} we may apply the same
reflection argument. These choices give distinct homogeneity
exponents, so the resulting solutions are not scalar multiples
of one another. In the following, we use only the solution
corresponding to phase $\pi$ at the equator.
\end{remark}

\section{Failure of the $C^1$-regularity of the natural gradient}

In this section we derive Corollary \ref{thm:counterexample} using Theorem \ref{thm:entire}. The shooting root \eqref{eq:root} determines a function $k_n(p)$.

\begin{lemma}\label{lem:branch}
	The root $k_n(p)>1$ is unique and continuous on $(1,\infty)$ with
	$k_n(2)=2$. Near $p=2$, multiply the angular factors from
	Theorem~\ref{thm:entire} by $n-1$ and call them $f_p$. Then
	\begin{equation}\label{eq:quadratic}
		f_2(\theta)=n\cos^2\theta-1,
		\qquad f_p\longrightarrow f_2\quad\hbox{in }C^2([0,\pi]),
		\quad\text{ as } p\to2.
	\end{equation}
\end{lemma}

\begin{proof}
	Existence and uniqueness were proved at the start of
	Section~\ref{sec:reflection} using Lemma~\ref{lem:shooting}.
	To prove continuity at $p_0>1$, put $k_0=k_n(p_0)$ and choose
	$\varepsilon>0$ with $k_0-\varepsilon>1$. Strict monotonicity gives
	$\Psi(p_0,k_0-\varepsilon)<\pi<\Psi(p_0,k_0+\varepsilon)$.
	By continuity of $\Psi$, these inequalities persist for $p$ close to $p_0$,
	so $k_0-\varepsilon<k_n(p)<k_0+\varepsilon$.

	For $p=2$ and $k=2$, consider the normalized angular function
	$f(\theta)=\frac{n\cos^2\theta-1}{n-1}$.
	Indeed, $r^2(n\cos^2\theta-1)=nx_n^2-|x|^2$ is harmonic and
	on $(0,\pi)$ the pair $(f,-f'/2)$ never vanishes.
	Since $f'<0$ on $(0,\pi/2)$ and $f(\pi/2)<0$, its continuous phase
	starting at zero reaches $\pi$ at the equator. Moreover, one has
	$\psi(\theta)=n\theta/(n-1)+O(\theta^3)=z_0(2,2)\theta+O(\theta^3)$
	so that the Lemma~\ref{lem:pole} identifies it with the constructed solution.

	This gives $\Psi(2,2)=\pi$ and $k_n(2)=2$. Finally, the claimed
	convergence in $C^2([0,\pi])$ follows from
	Lemma~\ref{lem:amplitude} and the reflection in
	Section~\ref{sec:reflection}. It also holds for the corresponding
	functions on $\Sph$. Indeed, in a pole chart write
	$h_p(s)=f_p(\arcsin s)$.
	Since $h_p'(0)=h_2'(0)=0$ it follows that
	$$\sup_{s>0}\frac{|h_p'(s)-h_2'(s)|}{s}
	\le\|h_p''-h_2''\|_\infty$$
	on a fixed small interval. The radial Hessian formula
	\eqref{eq:radial-hessian} therefore gives convergence in $C^2$
	at the poles as well as away from them.
\end{proof}

\begin{proposition}\label{prop:variation}
	The shooting root is differentiable at $p=2$, and
	\begin{equation}\label{eq:cn}
		k_n^\prime(2)=-\avgint_{\Sph}
		\frac{(n\omega_n^2-1)^2}{1+n(n-2)\omega_n^2}\dd S,
	\end{equation}
    where $\omega \in \Sph$.
	In particular, $k_2^\prime(2)=-1/2$ and
	$-1/2<k_n^\prime(2)<0$ for every $n\ge3$.
\end{proposition}

\begin{proof}
 Use the normalization of Lemma~\ref{lem:branch} and put
 $u_p(r\omega)=r^{k_n(p)}f_p(\omega)$.
 The normalized Euclidean infinity Laplacian of $u_p$ at
 $\omega\in\Sph$ is
 \[
 \Delta_\infty^N u_p(\omega)
 :=\frac{\langle D^2u_p(\omega)\,\nabla u_p(\omega),
                  \nabla u_p(\omega)\rangle}
          {|\nabla u_p(\omega)|^2}.
 \]
 Then we have
 \[
 \Delta_{\Sph}f_p+k_n(p)(k_n(p)+n-2)f_p
 +(p-2)\Delta_\infty^N u_p=0
 \qquad\text{on }\Sph.
 \]
 Multiplying the equation by $f_2$ and integrating by parts using
 $\Delta_{\Sph}f_2=-2n f_2$ yields
 \[
 (k_n(p)-2)(k_n(p)+n)\avgint_{\Sph}f_p f_2\dd S
 =-(p-2)\avgint_{\Sph}(\Delta_\infty^N u_p) f_2\dd S.
 \]
 For $p\ne2$ sufficiently close to $2$ we therefore can write
 the difference quotient as
 \begin{equation}\label{eq:projected}
 \frac{k_n(p)-k_n(2)}{p-2}=\frac{k_n(p)-2}{p-2}
 =-\frac{\avgint_{\Sph}(\Delta_\infty^N u_p) f_2\dd S}
 {(k_n(p)+n)\avgint_{\Sph}f_p f_2\dd S}.
 \end{equation}
  By Lemma~\ref{lem:branch} and the homogeneity of $u_p$,
 the gradients and Hessians of $u_p$ converge uniformly on $\Sph$
 to those of $u_2(x)=nx_n^2-|x|^2$.
 Since $|\nabla u_2|$ is bounded away from zero on $\Sph$,
 we obtain $\Delta_\infty^N u_p\to\Delta_\infty^N u_2$ uniformly there.
 Thus the numerator in \eqref{eq:projected} tends to
 $\avgint_{\Sph}(\Delta_\infty^N u_2) f_2\dd S$, while the denominator
 tends to $(n+2)\avgint_{\Sph}f_2^2\dd S>0$
 and hence $k_n^\prime(2)$ exists and
 \begin{equation}\label{eq:root deriv formula}
    k_n^\prime(2) = -\frac{\avgint_{\Sph}(\Delta_\infty^N u_2) f_2\dd S}
 {(k_n(2)+n)\avgint_{\Sph}f_2^2\dd S}
 \end{equation}
 Direct differentiation of $u_2$ gives
 \[
 \begin{aligned}
 \Delta_\infty^N u_2(\omega)
 &=2\frac{(n-1)^3\omega_n^2-(1-\omega_n^2)}
          {1+n(n-2)\omega_n^2}\\
 &=2(n-2)+2(n-1)\frac{f_2}{1+n(n-2)\omega_n^2}.
 \end{aligned}
 \]
 
 To evaluate the integrals, symmetry and $\sum_{i=1}^n\omega_i^2=1$
 first give $\avgint_{\Sph}\omega_n^2\dd S=1/n$.
 For brevity, let $a=\avgint_{\Sph}\omega_1^4\dd S$ and
 $b=\avgint_{\Sph}\omega_1^2\omega_2^2\dd S$.
 Rotational invariance gives
 \[
 a=\avgint_{\Sph}\bigl((\omega_1+\omega_2)/\sqrt2\bigr)^4\dd S
 =a/2+3b/2,
 \]
 since the terms odd in either coordinate have zero average
 and thus $a=3b$. Averaging $(\sum_{i=1}^n\omega_i^2)^2=1$ yields
 $na+n(n-1)b=1$, so $a=3/(n(n+2))$.
 Since $f_2=n\omega_n^2-1$, we obtain $\avgint_{\Sph}f_2\dd S=0$ and $\avgint_{\Sph}\omega_n^4\dd S=\frac{3}{n(n+2)}$, we obtain
 \begin{equation}\label{eq:moments}
 \avgint_{\Sph}f_2^2\dd S
 =n^2\avgint_{\Sph}\omega_n^4\dd S
   -2n\avgint_{\Sph}\omega_n^2\dd S+1
 =\frac{2(n-1)}{n+2}.
 \end{equation}
 
 Substituting these to \eqref{eq:root deriv formula} yields \eqref{eq:cn}.
 The integrand in \eqref{eq:cn} is nonnegative and is positive on a set
 of positive surface measure, so $k_n^\prime(2)<0$.
 For $n=2$, the denominator in \eqref{eq:cn} equals one,
 and $\avgint_{\Sph}f_2^2\dd S=1/2$, giving $k_2^\prime(2)=-1/2$.
 It remains to prove $k_n^\prime(2)>-1/2$ when $n\ge3$.
 
 The identities $\Delta_{\Sph}(\omega_n^2)=-2f_2$ and
 $|\nabla_{\Sph}(\omega_n^2)|^2=4\omega_n^2(1-\omega_n^2)$
 give, by integration by parts,
 \begin{equation}\label{eq:spherical-ibp}
 \avgint_{\Sph}f_2 g(\omega_n^2)\dd S
 =2\avgint_{\Sph}\omega_n^2(1-\omega_n^2)g'(\omega_n^2)\dd S
 \end{equation}
 for every smooth $g$ on $[0,1]$.
 Recall that
\[
-k_n^\prime(2)
=\avgint_{\Sph}\frac{f_2^2}{1+n(n-2)\omega_n^2}\dd S,
\]
and put
\[
R(\omega)=\frac{\omega_n^2(1-\omega_n^2)}
                {(1+n(n-2)\omega_n^2)^2}.
\]

First, apply \eqref{eq:spherical-ibp} to
$g(s)=(ns-1)/(1+n(n-2)s)$.
Since $f_2=n\omega_n^2-1$ and
$g'(s)=n(n-1)/(1+n(n-2)s)^2$, this gives
$-k_n^\prime(2)=2n(n-1)\avgint_{\Sph}R\dd S$.
Next, the identity
\[
1-\frac{f_2^2}{1+n(n-2)\omega_n^2}
=\frac{n^2\omega_n^2(1-\omega_n^2)}{1+n(n-2)\omega_n^2}
\]
and the definition of $R$ give
\begin{align*}
1+k_n^\prime(2)
&=\avgint_{\Sph}\left(1-\frac{f_2^2}{1+n(n-2)\omega_n^2}\right)\dd S\\
&=n^2\avgint_{\Sph}\bigl(1+n(n-2)\omega_n^2\bigr)R\dd S\\
&=n^2\avgint_{\Sph}R\dd S
  +n^3(n-2)\avgint_{\Sph}\omega_n^2R\dd S.
\end{align*}
Subtracting the preceding formula for $-k_n^\prime(2)$
from this formula for $1+k_n^\prime(2)$ now yields
\begin{align*}
1+2k_n^\prime(2)
&=\bigl(n^2-2n(n-1)\bigr)\avgint_{\Sph}R\dd S
  +n^3(n-2)\avgint_{\Sph}\omega_n^2R\dd S\\
&=n(n-2)\left(n^2\avgint_{\Sph}\omega_n^2R\dd S
                -\avgint_{\Sph}R\dd S\right).
\end{align*}
A second use of \eqref{eq:spherical-ibp}, now with
$g(s)=s(1-s)/(1+n(n-2)s)^2$, gives
\[
3\avgint_{\Sph}R\dd S-(n+4)\avgint_{\Sph}\omega_n^2R\dd S
=4n(n-2)\avgint_{\Sph}
\frac{\omega_n^2(1-\omega_n^2)R}{1+n(n-2)\omega_n^2}\dd S.
\]
Since
$4n(n-2)\omega_n^2/(1+n(n-2)\omega_n^2)
\le1+n(n-2)\omega_n^2$
is equivalent to $(n(n-2)\omega_n^2-1)^2\ge0$, and
$0\le1-\omega_n^2\le1$, the right side is at most
$\avgint_{\Sph}R\dd S+n(n-2)\avgint_{\Sph}\omega_n^2R\dd S$.
Therefore, for $n\ge3$,
\[
\avgint_{\Sph}\omega_n^2R\dd S
\ge\frac{2}{n^2-n+4}\avgint_{\Sph}R\dd S
>\frac{1}{n^2}\avgint_{\Sph}R\dd S.
\]
Here $\avgint_{\Sph}R\dd S>0$, since $R>0$ when $0<\omega_n^2<1$,
and $2n^2>n^2-n+4$ for $n\ge3$.
Thus $1+2k_n^\prime(2)>0$, proving the remaining bound.
\par
\end{proof}

\begin{proof}[Proof of Corollary~\ref{thm:counterexample}]
	Fix $n\ge3$ and use the solutions on the branch $k_n(p)$.
	Since $k_n$ is differentiable at $2$ one has
	\begin{equation}\label{eq:k-beta-expansion}
		\begin{split}
			k_n(p)&=2+k_n^\prime(2)(p-2)+o(|p-2|),
			\quad \text{and}\\
			\beta_n(p):=\frac p2(k_n(p)-1)
			&=1+\left(\frac12+k_n^\prime(2)\right)(p-2)
			+o(|p-2|).
		\end{split}
	\end{equation}
	Proposition~\ref{prop:variation} gives $k_n^\prime(2)<0$ and
	$1/2+k_n^\prime(2)>0$.
	Thus, after choosing $\delta_n\in(0,1)$ small enough, we have
	$k_n(p)>2, 0<\beta_n(p)<1$ and $2-\delta_n<p<2$.

	Choose the pole normalization $f_p(0)=1$ in the solution
	$u_p(r\omega)=r^{k_n(p)}f_p(\theta)$, where
	$\theta=\arccos\omega_n$,
	and write $Q_p=k_n(p)^2f_p^2+(f_p')^2$ and
	$W_p=V_p(\nabla u_p)$.
	By \eqref{eq:gradient},
	\begin{equation}\label{eq:V-angular}
		|W_p(r\omega)|=r^{\beta_n(p)}Q_p(\omega)^{p/4}.
	\end{equation}
	Positivity and continuity of $Q_p$ on the sphere prove
	\eqref{eq:pointwise}. For each fixed $\omega$, we have
	$|W_p(r\omega)|/r\to\infty$ as $r\to0$.
	Thus $W_p$ is neither locally Lipschitz at zero nor of class $C^1$.

	For the oscillation, put
	\[
	\Phi_{W_p}(r)=\left(
	\avgint_{B_r}
	\left|W_p-\avgint_{B_r}W_p\dd x\right|^2\dd x
	\right)^{1/2}.
	\]
	Homogeneity gives $W_p(rx)=r^{\beta_n(p)}W_p(x)$.
	Changing variables $x=ry$ in the average yields
	$\avgint_{B_r}W_p\dd x
	=r^{\beta_n(p)}\avgint_{B_1}W_p\dd x$,
	and the same change of variables in the oscillation gives
	\begin{equation}\label{eq:excess-constant}
		\Phi_{W_p}(r)=r^{\beta_n(p)}\Phi_{W_p}(1).
	\end{equation}
	The factor $\Phi_{W_p}(1)$ is positive: $W_p$ is continuous,
	vanishes at zero, and is nonzero elsewhere, so it cannot be
	constant almost everywhere on $B_1$. Hence
	$\Phi_{W_p}(\tau r)/(\tau\Phi_{W_p}(r))
	=\tau^{\beta_n(p)-1}\to\infty$ as $\tau\to0$,
	which contradicts \eqref{eq:failure-linear-decay}
	for every finite $C_0$.

	It remains to verify the asserted regularity of $u_p$ itself.
	Theorem~\ref{thm:entire} gives $C^2$ regularity off zero.
	Homogeneity and boundedness of the angular derivatives give
	the equations
	$|\nabla u_p(x)|\le C|x|^{k_n(p)-1}$ and
	$|D^2u_p(x)|\le C|x|^{k_n(p)-2}$ whenever $x\ne0$.
	As $k_n(p)>2$, the gradient is differentiable at zero with
	derivative zero, and the Hessian tends to zero there.
	Therefore $u_p\in C^2(\R^n)$ and $D^2u_p(0)=0$.
\end{proof}

\begin{remark}
	For $n=2$, the same construction gives $k_2^\prime(2)=-1/2$,
	and therefore $\beta_2'(2)=0$.
	Thus the above argument gives no planar counterexample. This agrees with the scalar planar
	regularity results cited in the introduction.
\end{remark}

\appendix
\section*{Appendix}

We include the coordinate computations to fix signs and normalizations.
For $u(r\omega)=r^kf(\theta)$, where $\theta=\arccos\omega_n$,
radial and angular differentiation away from the axis give
$\nabla u=r^{k-1}(kf\omega+f'e_\theta),
|\nabla u|^2=r^{2k-2}Q$ and $Q=k^2f^2+(f')^2$.
For a vector field $r^a(b(\theta)\omega+c(\theta)e_\theta)$, the polar
divergence formula is
\[
\diver\big(r^a(b\omega+ce_\theta)\big)
=r^{a-1}\left((a+n-1)b+
\sin^{2-n}\theta\,\frac{d}{d\theta}
\big(\sin^{n-2}\theta\,c\big)\right).
\]
Indeed, the radial part is
$r^{1-n}\partial_r(r^{n-1+a}b)$, and the angular part is
$r^{a-1}\sin^{2-n}\theta\,\partial_\theta(\sin^{n-2}\theta\,c)$.
Apply this with $a=(k-1)(p-1)$, $b=kfQ^{(p-2)/2}$, and
$c=Q^{(p-2)/2}f'$ to obtain
\[
\Delta_pu=r^{(k-1)(p-1)-1}
\left[\sin^{2-n}\theta\,
\big(\sin^{n-2}\theta\,Q^{(p-2)/2}f'\big)'
+\Lambda(p,k)Q^{(p-2)/2}f\right].
\]
This yields \eqref{eq:angular}. When $n=2$ the weight is one and the
formula applies on each of the two semicircles between the poles.

To derive the phase system, write $c=\cos\psi$, $s=\sin\psi$,
$h=\rho'/\rho$, and $v=\psi'$. Differentiation of $f=\rho c$ and
the desired relation $f'=-k\rho s$ give
$ch-sv=-ks$.
Since $Q=k^2\rho^2$, the logarithmic derivative of $Q^{(p-2)/2}$ is
$(p-2)h$. Substituting
$f''=-k\rho(hs+vc)$ in \eqref{eq:angular}, and dividing by the positive
factor $k\rho Q^{(p-2)/2}\sin^{n-2}\theta$, gives
$(p-1)sh+cv=\big((p-1)k+n-p\big)c-(n-2)\cot\theta\,s$.
The determinant of this system is
$c^2+(p-1)s^2=D_p(\psi)>0$. Solving it gives exactly
$v=F$ and $h=H$ in \eqref{eq:phase}--\eqref{eq:amplitude}.

Conversely, if $\psi$ solves \eqref{eq:phase} and $\rho$ is given by
\eqref{eq:profile}, the same linear system gives
$f'=-k\rho\sin\psi$ for $f=\rho\cos\psi$, and then
\eqref{eq:angular}.

\section*{Acknowledgments}
Generative AI was used in preparing this manuscript. The authors take
full responsibility for the final content.

\bigskip

\noindent
\begin{minipage}{\textwidth}
  \small
  \raggedright
  Department of Geoinformatics and Cartography,\\
  Finnish Geospatial Research Institute, Finland\\[2pt]
  \textit{Email address:}
  \href{mailto:erno.kauranen@maanmittauslaitos.fi}
       {\nolinkurl{erno.kauranen@maanmittauslaitos.fi}}
\end{minipage}

\medskip

\noindent
\begin{minipage}{\textwidth}
  \small
  \raggedright
  Department of Mathematics and Statistics,\\
  University of Jyv\"askyl\"a,\\
  PO Box 35, 40014 Jyv\"askyl\"a, Finland\\[2pt]
  \textit{Email address:}
  \href{mailto:jarkko.siltakoski@jyu.fi}
       {\nolinkurl{jarkko.siltakoski@jyu.fi}}
\end{minipage}


\begin{thebibliography}{99}
	\bibitem{Aronsson1988}
	G. Aronsson,
	\emph{On certain $p$-harmonic functions in the plane},
	Manuscripta Math. \textbf{61} (1988), 79--101.

	
	\bibitem{BalciBehnDieningStorn2026}
	A. Balci, L. Behn, L. Diening, and J. Storn,
	\emph{Examples of $p$-harmonic maps},
	SIAM J. Math. Anal. \textbf{58} (2026), no.~1, 260--275.
	
	
	\bibitem{BalciDieningWeimar2020}
	A. Kh. Balci, L. Diening, and M. Weimar,
	\emph{Higher order Calder\'on--Zygmund estimates for the $p$-Laplace equation},
	J. Differential Equations \textbf{268} (2020), no.~2, 590--635.

	
	\bibitem{IwaniecManfredi1989}
	T. Iwaniec and J. J. Manfredi,
	\emph{Regularity of $p$-harmonic functions on the plane},
	Rev. Mat. Iberoam. \textbf{5} (1989), no.~1, 1--19.
	
	
	\bibitem{Krol1973}
	I. N. Kr\'ol,
	\emph{The behaviour of the solutions of a certain quasilinear equation
		near zero cusps of the boundary},
	Trudy Mat. Inst. Steklov. \textbf{125} (1973), 140--146, 233
	(in Russian).
	
	\bibitem{Lindqvist2006}
	P. Lindqvist,
	\emph{Notes on the $p$-Laplace Equation},
	Report 102, Department of Mathematics and Statistics,
	University of Jyv\"askyl\"a, 2006.
	
	\bibitem{Manfredi1988}
	J. J. Manfredi,
	\emph{$p$-harmonic functions in the plane},
	Proc. Amer. Math. Soc. \textbf{103} (1988), 473--479.
	
	\bibitem{PorrettaVeron2009}
	A. Porretta and L. V\'eron,
	\emph{Separable $p$-harmonic functions in a cone and related quasilinear
		equations on manifolds},
	J. Eur. Math. Soc. \textbf{11} (2009), no.~6, 1285--1305.



 
\bibitem{LlorenteManfrediTroyWu2019}
J. G. Llorente, J. J. Manfredi, W. C. Troy, and J.-M. Wu,
\emph{On $p$-harmonic measures in half-spaces},
Ann. Mat. Pura Appl. (4) \textbf{198} (2019), no.~4, 1381--1405.



\bibitem{Tolksdorf1983}
P. Tolksdorf,
\emph{On the Dirichlet problem for quasilinear equations
in domains with conical boundary points},
Comm. Partial Differential Equations \textbf{8} (1983),
no.~7, 773--817.


\bibitem{Tkachev2020}
V. G. Tkachev,
\emph{New explicit solutions to the $p$-Laplace equation
based on isoparametric foliations},
Differential Geom. Appl. \textbf{70} (2020), 101629.

 
\end{thebibliography}
\end{document}